\documentclass[pdflatex,sn-mathphys-num]{sn-jnl}

\def\R{\mathbb R}
\def\eps{\varepsilon}

\def\prox{\mbox{\rm prox}}
\usepackage{graphicx}%
\usepackage{multirow}%
\usepackage{amsmath,amssymb,amsfonts}%
\usepackage{amsthm}%
\usepackage{mathrsfs}%
\usepackage[title]{appendix}%
\usepackage{xcolor}%
\usepackage{textcomp}%
\usepackage{manyfoot}%
\usepackage{booktabs}%
\usepackage{algorithm}%
\usepackage{algorithmicx}%
\usepackage{algpseudocode}%
\usepackage{listings}%

\theoremstyle{thmstyleone}%
\newtheorem{theorem}{Theorem}
\newtheorem{proposition}[theorem]{Proposition}%
\newtheorem{corollary}[theorem]{Corollary}
\newtheorem{lemma}[theorem]{Lemma}
\theoremstyle{thmstyletwo}%
\newtheorem{example}{Example}%
\newtheorem{remark}{Remark}%

\newtheorem{definition}{Definition}%
\begin{document}

\title[Novel Dynamical Systems with Finite-Time and Predefined-Time Stability for Generalized Inverse Mixed Variational Inequality Problems]{Novel Dynamical Systems with Finite-Time and Predefined-Time Stability for Generalized Inverse Mixed Variational Inequality Problems}


\author[1]{\fnm{Nam Van} \sur{Tran}}\email{namtv@hcmute.edu.vn}

\affil[1]{\orgdiv{Faculty of Applied Sciences}, \orgname{Ho Chi Minh City  University of Technology and Engineering}, \orgaddress{\street{Vo Van Ngan}, \city{Ho Chi Minh City}, 
\country{Vietnam}}}


\abstract{This paper investigates a class of generalized inverse mixed variational inequality problems (GIMVIPs), which consist in finding a vector $\overline{w}\in \R^d$ such that
\[
F(\bar w)\in \Omega \quad \text{and} \quad 
\langle h(\bar w), v-F(\bar w) \rangle + g(v)-g(F(\bar w)) \ge 0,
\quad \forall v\in \Omega,
\]
where \(h,F:\R^d\to\R^d\) are single-valued operators, \(g:\Omega\to\R\cup\{+\infty\}\) is a proper function, and \(\Omega\) is a closed convex set.

Two novel continuous-time dynamical systems are proposed to study the finite-time and predefined-time stability of solutions to GIMVIPs in finite-dimensional Hilbert spaces. Under suitable assumptions on the involved operators and model parameters, Lyapunov-based techniques are employed to establish finite-time and predefined-time convergence of the generated trajectories.

Although both dynamical systems exhibit accelerated convergence, the settling time of the finite-time stable system depends on the initial condition, whereas the predefined-time stable system admits a uniform upper bound on the convergence time that is independent of the initial state and can be explicitly prescribed through user-selected parameters. Moreover, by applying a forward Euler discretization to the continuous-time dynamics, a proximal point-type iterative algorithm is derived, and its fixed-time convergence property is rigorously analyzed. Numerical experiments are provided to illustrate the effectiveness and advantages of the proposed methods.

}





\maketitle

\section{Introduction}
Let \(\langle \cdot,\cdot \rangle\) and \(\|\cdot\|\) denote the inner product and the associated norm in \(\mathbb{R}^d\), respectively. 
This paper is concerned with a class of inverse equilibrium problems known as the \emph{generalized inverse mixed variational inequality problem} (GIMVIP). 
Given a nonempty closed and convex set \(\Omega \subset \mathbb{R}^d\), the GIMVIP consists in finding a vector \(\bar w \in \mathbb{R}^d\) such that
\begin{equation}\label{GIMVIP}
F(\bar w)\in \Omega \quad \text{and} \quad
\langle h(\bar w), v - F(\bar w) \rangle + g(v) - g(F(\bar w)) \ge 0,
\quad \forall v \in \Omega .
\end{equation}

Problem~\eqref{GIMVIP} provides a unifying framework that encompasses several important models in variational analysis. 
If \(h\) is the identity mapping, then \eqref{GIMVIP} reduces to the inverse mixed variational inequality problem introduced in~\cite{CLi}. 
When \(g \equiv 0\), it collapses to the inverse variational inequality problem (IVIP), originally proposed by He and Liu~\cite{BSHe}. 
On the other hand, if \(F\) is the identity operator, \eqref{GIMVIP} becomes a mixed variational inequality problem (MVIP) and further reduces to the classical variational inequality problem (VIP) when \(g \equiv 0\).

Variational inequality problems constitute a fundamental modeling framework for equilibrium phenomena arising in optimization, saddle-point problems, Nash equilibria in noncooperative games, and fixed-point theory; see, for instance,~\cite{BlumOettli94,Cav}. 
Beyond direct formulations, \emph{inverse} variational inequality problems have received increasing attention due to their ability to infer unknown system parameters from observed equilibrium states. 
Such inverse models naturally arise in applications where equilibrium behavior is observable, whereas the underlying cost structures, control parameters, or governing operators are not directly accessible.

In recent years, inverse variational inequality models have been successfully applied in traffic network analysis, economic equilibrium modeling, and telecommunication systems. 
To further enhance modeling flexibility, inverse mixed variational inequality problems were proposed by incorporating nonsmooth and composite terms, allowing richer representations of practical constraints. 
Subsequent developments led to generalized inverse mixed variational inequality formulations, which subsume many existing inverse models and provide a comprehensive framework for inverse equilibrium analysis. 
Motivated by these advances, the present work focuses on GIMVIPs as a versatile and unified inverse equilibrium model.

Alongside model generalization, the design of efficient solution methods remains a central challenge. 
Continuous-time dynamical systems have emerged as a powerful tool for solving variational inequalities and related equilibrium problems. 
By reformulating equilibrium conditions as dynamical processes, one can exploit Lyapunov stability theory to analyze convergence properties. 
Compared with purely discrete-time schemes, continuous-time approaches offer conceptual simplicity, potential for real-time implementation, and a natural interpretation from a feedback control perspective.

Most existing dynamical system approaches for variational and inverse variational inequality problems guarantee asymptotic or exponential stability. 
Specifically, dynamical system techniques have been applied to optimization problems~\cite{BEN,COR,GAR,HE,LIEN,ROMEO}, sparse optimization~\cite{GARG4,HE,YU}, control and multi-agent systems~\cite{ZUO}, mixed variational inequality problems~\cite{Garg21,JU3}, as well as variational inequality and equilibrium problems~\cite{JU4,JU5}. 
More recently, Nam \emph{et al.}~\cite{NAM_HAI} investigated finite-time stability for generalized monotone inclusion problems, while fixed-time stable dynamical systems for such inclusions were developed in~\cite{NHVA}. 
In addition, Ju \emph{et al.}~\cite{XXJU} employed continuous-time dynamical systems to establish fixed-time stability for pseudo-monotone mixed variational inequality problems.

For inverse and inverse quasi-variational inequality problems, dynamical system methods have also been explored; see, for example,~\cite{ZOU2}. 
In particular, Dey and Reich~\cite{Dey} proposed a first-order continuous-time dynamical system for inverse quasi-variational inequality problems and established global exponential stability. 
Subsequently, Vuong and Thanh~\cite{VuongThanh} analyzed a time-discretized version of a related system and obtained linear convergence results under a more general moving set condition. 
More recently, Tangkhawiwetkul~\cite{JIT} studied generalized inverse mixed variational inequality problems using dynamical system techniques and derived asymptotic as well as exponential stability results.

However, in many applications requiring rapid, predictable, and time-critical convergence, asymptotic or exponential stability may be insufficient. 
Finite-time stability, introduced in~\cite{BHAT}, guarantees convergence within a finite settling time, although this time typically depends on the initial condition. 
To address this limitation, Polyakov~\cite{polyakov} introduced the concept of fixed-time stability, for which the convergence time admits a uniform upper bound independent of the initial state. 
Despite their practical relevance, finite-time and fixed-time convergence properties have not yet been systematically investigated for generalized inverse mixed variational inequality problems.

Beyond finite-time and fixed-time stability, an even stronger notion, known as \emph{predefined-time stability}, was introduced by Sánchez-Torres \emph{et al.}~\cite{Sanchez2014}. 
In contrast to fixed-time stability, where convergence is guaranteed within an unknown but bounded time, predefined-time stability allows the convergence time to be explicitly prescribed in advance through user-selected parameters and remains independent of the initial condition. 
This feature provides enhanced flexibility and autonomy in system design, enabling practitioners to enforce convergence within a desired time horizon that meets practical control or performance requirements.

Predefined-time stability has been actively studied in control and synchronization contexts. 
For instance, Anguiano-Gijón \emph{et al.}~\cite{Ato} developed Lyapunov-function-based control schemes for predefined-time synchronization of chaotic systems, while Liu \emph{et al.}~\cite{ALIU} proposed novel predefined-time stability lemmas and applied them to synchronization problems. 
These studies demonstrate the effectiveness of predefined-time stability in facilitating coordination and information exchange among interconnected systems. 
Nevertheless, the incorporation of predefined-time stability into optimization and inverse equilibrium problem frameworks remains largely unexplored.

Motivated by the above observations, this paper develops novel first-order continuous-time dynamical systems for solving generalized inverse mixed variational inequality problems and establishes their finite-time, fixed-time, and predefined-time stability properties. 
By applying a forward Euler discretization, we further derive a proximal point-type iterative algorithm and show that the generated sequence converges to the unique solution of the GIMVIP within a prescribed time horizon. 
The proposed framework not only generalizes existing dynamical system approaches for inverse variational inequality problems but also provides efficient and user-controllable solution methods for generalized inverse equilibrium models.

The remainder of the paper is organized as follows.
Section~\ref{Preliminaries} introduces preliminary material, including notions of finite-time, fixed-time, and predefined-time stability, and establishes the correspondence between GIMVIPs and the equilibrium points of an associated proximal dynamical system.
Section~\ref{finite stability} investigates the finite-time stability of a newly proposed dynamical system.
Section~\ref{sec3} is devoted to the fixed-time and predefined-time stability analysis of an alternative system.
Section~\ref{sec4}  studies the consistency of its time discretization, ensuring convergence within a prescribed number of iterations.
Numerical experiments illustrating the effectiveness of the proposed methods are reported in Section~\ref{num}.
Finally, Section~\ref{sec6} concludes the paper with a summary of the main results.

	\section{Preliminaries} \label{Preliminaries}
	
	This section introduces the basic notation, definitions, and auxiliary results that will be used throughout the paper. Throughout the paper, we let $\Omega$ be a closed, convex, nonempty subset of $\R^d$, we denote by $B( \overline{w}, \delta)$  the open ball with center $\overline{w}$ and radius $\delta>0$.

	\subsection{Notions and auxiliary results}

	\begin{definition}
	    Let   $F: \R^d \to \R^d$ be a single-valued mapping and $g~:\Omega \longrightarrow \R\cup\{+\infty\}$ be a proper, convex, lower semicontinuos (l.s.c.). We say that $(F, h))$ is $\mu$-strongly monotone couple  with modulus  $\lambda\geq 0$ if for any $u, v\in \R^d$ we have 
    \begin{equation*}\label{mon}
       \langle F(w)-F(v), h(w)-h(v)\rangle \geq \mu \|w-v\|^2.
    \end{equation*}
	When $\mu=0$ we say that $(F,h)$ is \emph{monotone couple}. 
       \end{definition} 
    If $g$ is the identity mapping, the notion of strongly monotone couple reduces to the classical one of (strong) monotonicity. In fact, $F$ is said to be $\mu$-strongly monotone if 
     \begin{equation*}\label{mon}
       \langle F(w)-F(v), w-v\rangle \geq \mu \|w-v\|^2\quad  \forall w, v\in \R^d. 
    \end{equation*}
 
    In addition, if $F$ is strongly monotone with modules $\mu\geq 0$ then we have 
    \begin{equation*}
        \|F(w)-F(v)\|\geq \mu \|w-v\| \quad \forall  w, v\in \R^d.
    \end{equation*} 
	
\begin{definition}
    Let $F: \R^d  \to \R^d$ be a single-valued mapping. The mapping $F$ is said to be 
    \begin{enumerate}
        \item[(i)] $\beta$-\emph{Lipschitz continuous} on \(\R^d\) if 
    \begin{equation*}
        \|F(w)-F(v)\|\leq \beta\|w-v\| \quad \forall w, v\in \R^d.
    \end{equation*} 
    When $\beta=1$, we say that $F$ is {\it nonexpansive.} 
    \item [(ii)] $\rho$-\emph{cocoercive} on $\R^d$ ($\rho >0$) if   
    \begin{equation*}
         \langle F(w)-F(v), w-v\rangle \geq \rho\|F(w)-F(v)\|^2 \quad \forall w, v\in \R^d. 
    \end{equation*}
    \end{enumerate}
\end{definition}
\begin{remark}
    It is worth noting that if $F$ is $\rho$-cocoercive then $F$ is Lipschitz continuous with Lipschitz constant $\beta=\frac{1}{\rho}$. 
\end{remark}	

	\noindent We now recall an important notion of proximal operator, which is a useful tool in studying variational analysis. Let $g:\Omega \longrightarrow \R\cup\{+\infty\}$ be a l.s.c. function.     For every $w\in \R^d$, we define 
    \begin{equation*}
        \prox_\Omega^g(w)=\mbox{\rm argmin}_{v\in \Omega} \left\{\|g(v)+\frac{1}{2}\|w-v\|^2\right\}.
    \end{equation*}
and called the \emph{proximal operator} of $g$ on $\Omega$. 
It is worth noting that when $g$ is the indicator function of a set $\Omega$, that is,
\[
g(w) \equiv \iota_\Omega(w)=
\begin{cases}
0, & \text{if } w\in \Omega,\\
+\infty, & \text{if } w\notin \Omega, 
\end{cases}
\]
the proximal operator reduces to the projection operator onto $C$. Hence, in some works, such as \cite{JIT}, the proximal operator is also referred to as a generalized projection.

The proximal operator admits the following properties.
\begin{lemma} \cite{Combet} 
Let $\Omega$ be a nonempty, closed, convex set in \(\R^d\). Let $w,  v \in \R^d$. Then  one has
\begin{enumerate}
    \item[(i)] $\overline{w}=\prox_\Omega^g(w)  \iff \langle \overline{w}-u, v-\overline{w}\rangle +g(v)-g\left(\overline{w}\right) \geq 0 \quad \forall v\in \Omega$. 
    \item [(ii)] $\|\prox_\Omega^g(w)-\prox_\Omega^g(v)\|^2 \leq \langle w-v, \prox_\Omega^g(w)-\prox_\Omega^g(v)\rangle $.
    \item [(iii)] $\prox_\Omega^g$ is nonexpansive, i.e. $\|\prox_\Omega^g(w)-\prox_\Omega^g(v)\| \leq  \|w-v\|.$ 
\end{enumerate}
\end{lemma}
We refer the readers to \cite{Combet} for more details in this operator. 

Throughout the sequel, $\mathrm{Zer}(\Xi)$ denotes the solution set of the generalized equation $\Xi(w)=0$, and $\dot V$ stands for the derivative of $V$.
	\subsection{Equilibrium points and stability}
\begin{definition} \cite{Pappaladro02} Consider the following differential equation system
	\begin{equation}\label{JJS1}
		\dot{w}(t) = \Xi(w(t)) \quad t \ge 0, 
	\end{equation}
	where $\Xi: \mathbb{R}^n \rightarrow \mathbb{R}^n$ is a continuous mapping and $u:[0, +\infty)\to \R^d$.

		A point $\overline{w}\in \R^d$ is said to be an equilibrium point of \eqref{JJS1} if
$\Xi(\overline{w})=0$.

An equilibrium point $\overline{w}$ is called \emph{Lyapunov stable} if, for any
$\epsilon>0$, there exists $\kappa>0$ such that, for every
$u_0 \in B(\overline{w}, \kappa)$, the corresponding solution $w(t)$ of the dynamical
system \eqref{JJS1} with $w(0)=u_0$ exists and satisfies
$w(t)\in B(\overline{w}, \epsilon)$ for all $t>0$.

An equilibrium point $\overline{w}$ is called \emph{finite-time stable} if it is
Lyapunov stable and there exist a neighborhood $B(\overline{w}, \delta)$ of $\overline{w}$ and
a settling-time function
$T: B(\overline{w}, \delta)\setminus \{\overline{w}\}\to (0,\infty)$ such that, for any
$w(0)\in B(\overline{w}, \delta)\setminus \{\overline{w}\}$, the solution of \eqref{JJS1}
satisfies $w(t)\in B(\overline{w}, \delta)\setminus \{\overline{w}\}$ for all
$t\in [0,T(w(0)))$ and
\[
\lim_{t\to T(w(0))} w(t)=\overline{w}.
\]

It is called \emph{globally finite-time stable} if it is finite-time stable
with $B(\overline{w}, \delta)=\R^d$.

Finally, $\overline{w}$ is called \emph{fixed-time stable} if it is globally
finite-time stable and the settling-time function satisfies
\[
\sup_{w(0)\in \R^d} T(w(0)) < \infty.
\]	
If we can adjust the system parameters such that the settling time in a fixed-time stable dynamical system can be predicted, then the equilibrium is called a predefined-time stable equilibrium.  
\end{definition}

In \cite{BHAT}, Bhat and Bernstein provided a sufficeint condition for the finite-time stability of solutions of a differential equation system \eqref{JJS1}, which is frequently used to verify the fintie and fixed time stability.

\begin{theorem}\cite{BHAT} \label{lm1-finite}	(Lyapunov condition for finite-time stability).  Let $\overline{w}\in \R^d$ be an equilibrium point of \eqref{JJS1}, and $D \subseteq \R^d$  be a neighborhood of  \(\overline{w}\). Assume that there exist a continuously differentiable function $V : D\rightarrow \R$,  and an open neighborhood $U\subseteq D$ of $\overline{w}$ such that
\begin{equation*}
	\dot{V}(w) \leq -M \,\big(V(w)\big)^r \quad \forall w\in U\setminus \{\overline{w}\},
\end{equation*}
where $M>0$ and $r\in (0,1).$  Then \(\overline{w}\) is a finite-time stable equilibrium point of \eqref{JJS1}. Moreover, the settling time $T$ satisfies 

\begin{equation*}
	T(w(0)) \leq \dfrac{V(w(0))^{1-r}}{M(1-r)}
\end{equation*}
for any $w(0) \in U$.  

Furthermore,  the equilibrium point \(\overline{w}\) of \eqref{JJS1} is globally finite-time stable when  $D =\R^d$.
\end{theorem}
As for the fixed time stability of the solution of the dynamical system \eqref{JJS1},   in  \cite{polyakov},  Polyakov gave a sufficient condition to guarantee the fixed time stability. This sufficient condition is recalled below. 

 	\begin{theorem} \label{lm1}	(Lyapunov condition for fixed-time stability). An equilibrium point $\overline{w}$ of \eqref{JJS1} is fixed time stable if  there exists 	a radially unbounded, continuously differentiable function $V~: \R^d\rightarrow \R$ such that
		\begin{equation*}
			V(\overline{w}) = 0, \mbox{  }V(w) > 0
		\end{equation*}
		for all $w\in  \R^d \setminus \{\overline{w}\}$ and
		\begin{equation} \label{inq of Lyapunov} 
			\dot{V}(w) \leq  -\left(A_1V(w)^{r_1} +A_2V(w)^{r_2}\right)
		\end{equation}
		for all $w\in \R^d \setminus \{\overline{w}\}$, where  $A_1,A_2,r_1,r_2> 0$ such that $r_1 <1$ and $r_2 >1$. Moreover, the settling time function satisfied 
		\begin{equation*}\label{time1}
			T(w(0)) \leq \dfrac{1}{A_1(1-r_1)} +\dfrac{1}{A_2(r_2-1)}
		\end{equation*}
		for any $w(0) \in \R^d$.  
		In addition, choosing $r_1=\left(1-\frac{1}{2\chi}\right), r_2=\left(1+\frac{1}{2\chi}\right),$ with $\chi>1$  in \eqref{inq of Lyapunov}, the settling time function is given by
		\begin{equation*}\label{time2}
			T(w(0))\leq T_{\max } =\dfrac{\pi\chi}{\sqrt{A_1A_2}}. 
		\end{equation*}		
	\end{theorem}

\subsection{A nominal dynamical system }
In \cite{JIT} the author proposed the following dynamical system for solving the GIMVIPs: 
$$\dot{w}=\kappa \Big(F(w)-\mbox{\rm prox}_{\Omega}^{\gamma g}\big(F(w)-h(w)\big)\Big), \quad \kappa>0.$$	
\noindent In the present paper, we generalize this model by allowing $\kappa$ to be a function of $w$. Accordingly, we consider the following dynamical system for GIMVIPs:
\begin{equation}\label{hdl2}
    \dot{x}=\tau(w)\Xi(w)
\end{equation}
with \[\Xi(w):=F(w)-\mbox{\rm prox}_{\Omega}^{\gamma g}\big(F(w)-h(w)\big),\] where \(\gamma>0\) is a parameter, $\tau (w)>0$ for all $w\in \R^d\setminus \mbox{\rm Zer}(\Xi)$.  
\\
For notational convenience, we denote $$B(w):=\prox^{\gamma g}_{\Omega}  \big(F(w)-h(w)\big).$$
 The following proposition gives a relationship between the equilibrium point of the dynamical system \eqref{hdl2} and the zero point of $\Xi$. 
\begin{proposition}\cite{JIT}  Let $h, F~:  \R^d \longrightarrow \R^d$ be   single-valued mappings. Let $ g~: \Omega \longrightarrow \mathbb{R}\cup \{+\infty\}$  be a proper, convex, and lower semicontinuous function. Then 
   $\overline{w}\in \R^d$  is a solution to GIMVIP \eqref{GIMVIP} if and only if \(\overline{w}\) is a zero point  of $\Xi$, i.e, \[\Xi(\overline{w})\equiv F(\overline{w})-\mbox{\rm prox}^{\gamma g}_\Omega\left(F(\overline{w}\right)-h\left(\overline{w}\right))=0.\]  
\end{proposition}

From this result, we directly obtain the following consequence.
\begin{corollary}\label{mlh1}
    Let $h, F~:  \R^d \longrightarrow \R^d$ and  $ g~: \Omega \longrightarrow \mathbb{R}\cup \{+\infty\}$ be a proper convex l.s.c function. Then, $\overline{w}\in \R^d$ is a solution to the GIMVIP \eqref{GIMVIP} if and only if it is an equilibrium point of the system \eqref{hdl2}. 
\end{corollary}

The following assumption is introduced to facilitate the analysis of finite-time and fixed-time stability of the GIMVIP \eqref{GIMVIP}.

\noindent ({\bf A})  Let $h~:  \R^d \longrightarrow \R^d $  be a Lipschitz continuous single-valued mapping with Lipschitz constant $\alpha$. Let   \(F: \R^d \longrightarrow \R^d\) be a single-valued mapping. Let $ g~: \Omega \longrightarrow \mathbb{R}\cup \{+\infty\}$  be a proper, convex, and lower semicontinuous function. Assume further that  
\begin{enumerate}
    \item [(i)] $h$ is $\lambda$-monotone on $\R^d$, and $(F, h)$ is $\mu$-strongly monotone on $\R^d$. 
    \item[(ii)] $F$ is $\rho$-cocoercive with constant $\rho>0.$
    \item [(iii)]   \(\sqrt{\beta^2+\alpha^2-2\mu}+\sqrt{1-2\lambda+\alpha^2}<1,\)
    where $\mu<\frac{\alpha^2+\beta^2}{2}, \lambda < \frac{1+\alpha^2}{2}, \beta=\frac{1}{\rho}$.
    
    \item[(iv)] $\rho>\sqrt{\alpha^2+\beta^2-2\mu}$. 
\end{enumerate}

\noindent{\bf Notation:} From now on we  denote $$\Gamma:=\sqrt{\beta^2+\alpha^2-2\mu}+\alpha,$$ and      $$\Lambda=\sqrt{\alpha^2+\beta^2-2\mu}.$$
We now present an example to show that there exist amppings and functions that satisfy Assumption ({\bf A}).  This example is adapted from the one in \cite{JIT}. 
\begin{example}\label{ex1} 
Let $d=1, \Omega=[0,+\infty)$. We define $h(w)=\frac{w}{2}, F(w)=\frac{3w}{4}, g(x)=w^2+2w+1$ for all $w\in \R$.  Then it is easy to see that $h$ is $\frac{1}{2}$-Lipschit continuous and $\frac{1}{2}$-strongly monotone, while $F$ is $\frac{3}{4}$-Lipschitz continuous and $\frac{4}{3}$-cocoercive, and $(F, h)$ is \(\frac{3}{8}\)-strongly monotone couple.  A straightforward verification shows that the assumpetion ({\bf A}) is satisfied for mappings $F, h, g$ with $\alpha=\frac{1}{2}=\lambda, \beta=\frac{3}{4}, \rho=\frac{4}{3}, \mu=\frac{3}{8}$. 
\end{example}
The theorem below presents a sufficient condition for the existence and uniqueness of solutions to the GIMVIP \eqref{GIMVIP}.

 \begin{theorem}\cite{JIT}\label{unique sol}
        Let $ h~:  \R^d \longrightarrow \R^d $  be a Lipschitz continuous single-valued mapping with Lipschitz constant $\alpha$. Let   \(F: \R^d \longrightarrow \R^d\) be a  $\beta-$ Lipschitz continuous single-valued mapping. Let $ g~: \Omega \longrightarrow \mathbb{R}\cup \{+\infty\}$  be a proper, convex, and lower semicontinuous function. If Conditions (i) and (iii) in Assumption ({\bf A}) are satisfied, then the GIMVIP \eqref{GIMVIP} has a unique solution.
    \end{theorem}

The following lemma lists some useful inequalities that are needed in the sequel.

\begin{lemma} \label{bdt}Let $h, F~:  \R^d \longrightarrow \R^d$ and  $ g~: \Omega \longrightarrow \mathbb{R}\cup \{+\infty\}$ be such that assumption ({\bf A}) is satisfied. Let $\overline{w}$ be the unique solution to the GIMVIP \eqref{GIMVIP}. Then the following holds.
\begin{itemize}
    \item [(i)] $\|\Xi(w_1) -\Xi(w_2)\| \leq \Gamma \|w_1-w_2\|$ for all $w_1, w_2\in \R^d$.
    \item[(ii)]  $\|B(w)-B\left(\overline{w}\right)\|\leq \big(\sqrt{\beta^2+\alpha^2-2\mu}\big)  \|w-\overline{w}\| =\Lambda \|w-\overline{w}\|$ for all $w\in \R^d$.  
    \item[(iii)] $\left(\rho-\Lambda \right)\|w-\overline{w}\| \leq \|\Xi(w)\|\leq \left(\sqrt{\beta^2-2\mu+\alpha^2}+\alpha\right) \|w-\overline{w}\|=\Gamma \|w-\overline{w}\|$ for all $w\in \R^d$.      
    \item[(iv)] $\langle w-\overline{w}, \Xi(w) \rangle      
        \geq (\rho-\Lambda) \|w-\overline{w}\|^2 >0 \ \forall w\in \R^d, \ w\neq \overline{w}$.
\end{itemize}
 
\end{lemma}
\begin{proof} Parts (i) and the second part of (iii) are proved in \cite{JIT}.

\noindent  (ii)  We have     \begin{align*}
        &\|B(w)-B\left(\overline{w}\right)\|\leq  \|F(w)-F(\overline{w})-\big(h(w)-h(\overline{w})\big)\|\\
        \leq & \sqrt{\|F(w)-F(\overline{w})\|^2+\|h(w)-h(\overline{w})\|^2-2\langle F(w)-F(\overline{w}), h(w)-h(\overline{w})\rangle} \\
        \leq & \sqrt{\beta^2\|w-\overline{w}\|^2+\alpha^2\|w-\overline{w}\|^2-2\mu \|w-\overline{w}\|^2} =\big(\sqrt{\alpha^2+\beta^2-2\mu}\big)\|w-\overline{w}\|.
    \end{align*} 

\noindent    (iii) We now prove the first inequality of Part (iii).  We have 
     \begin{align*}
         \|\Xi(w)\|= &\|\Xi(w)-A\left(\overline{w}\right)\|\\
         = &\|F(w)-F(\overline{w})-\big(B(w)-B\left(\overline{w}\right)\big)\| \geq \rho\|w-\overline{w}\|-\sqrt{\alpha^2+\beta^2-2\mu}\|w-\overline{w}\|\\
         \geq &(\rho-\Lambda)\|w-\overline{w}\|. 
     \end{align*}
 
 \noindent (iv)   Using the $\rho$-cocoercive property of $F$ and  Part (ii),  one has, for all $w\in \R^d$,
\begin{align*}
	&\langle w-\overline{w}, \Xi(w) \rangle\\
     =& \langle w-\overline{w}, F(w)-F\left(\overline{w}\right)-(B(w)-B\left(\overline{w}\right) \rangle\\
     =& \langle w-\overline{w}, F(w)-F\left(\overline{w}\right)\rangle -\langle w-\overline{w}, B(w)-B\left(\overline{w}\right) \rangle\\
    \geq  & \rho\|w-\overline{w}\|^2 - \|w-\overline{w}\|\|B(w)-B\left(\overline{w}\right)\|\\
    \geq  &\rho\|w-\overline{w}\|^2 - \Lambda\|w-\overline{w}\|^2\\  
     = & \left (\rho-\Lambda \right ) \|w-\overline{w}\|^2 >0. 
`	\end{align*}
   
\end{proof}


\section{Finite-time stability  analysis}\label{finite stability}
In this section, we will analyze the finite-time stability for the GIMVIP \eqref{GIMVIP}. By doing so, we propose a new first-order dynamical system associated with problem \eqref{GIMVIP} such that a solution of \eqref{GIMVIP} becomes an equilibrium point of the dynamical system. Under mild conditions for the parameters, we demonstrate that the proposed dynamical system is finite-time stable.

In order to obtain a finite-time stability of a solution of the GIMVIP \eqref{GIMVIP}, we now present a novel first-order dynamical system associated with the GIMVIP \eqref{GIMVIP}.  Our proposed dynamical system is as follows: 
\begin{equation}\label{newydynamicalsystem-finite} 
\dot{w}=-\tau  \dfrac{\Xi(w)}{\|\Xi(w)\|^{\frac{k-2}{k-1}}}, 
\end{equation} 
where 
$\tau>0$  is a scalar tuning gain, and $k>2$ is a design parameter. 

\noindent {\bf Convention:} For $w\in \R^d$ such that $\Xi(w)=0$, and for $k>2$ we use the following identity  
\begin{equation*}
   \left\|  \dfrac{\Xi(w)}{\|\Xi(w)\|^{\frac{k-2}{k-1}}}\right\|= \|\Xi(w)\|^{\frac{1}{k-1}}=0.
\end{equation*}
\begin{remark}
    When $k=2$, the dynamical system \eqref{newydynamicalsystem-finite} reduces to the one studied in \cite{JIT}, which was used to establish asymptotic and exponential convergence.
\end{remark}
The following result establishes the connection between the equilibrium points of the previously mentioned dynamical system and those of \eqref{hdl2}.

\begin{proposition}\label{pro3-finite}
A point  $\overline{w}\in \R^d$ is an equilibrium point of  \eqref{newydynamicalsystem-finite} if and only of it is also an equilibirum point of \eqref{hdl2}. 
\end{proposition}
\begin{proof} A point  $\overline{w}\in \R^d$ is an equilibirum point of  the dynamical system  \eqref{newydynamicalsystem-finite} if and only if 
$$\dfrac{\Xi\left(\overline{w}\right)}{\|\Xi\left(\overline{w}\right)\|^{\frac{k-2}{k-1}}}= 0 \iff   \dfrac{\|\Xi\left(\overline{w}\right)\|}{\|\Xi\left(\overline{w}\right)\|^{\frac{k-2}{k-1}}} =0.$$  This equation is equivalent to the following  
\begin{equation}\label{equ80}
	\|\Xi\left(\overline{w}\right)\|^{1-\frac{k-2}{k-1}}=0.
\end{equation}
Because $1-\frac{k-2}{k-1}=\frac{1}{k-1}>0$ for all $k>2$,  equality \eqref{equ80} holds if and only if  $$\Xi\left(\overline{w}\right)=0,$$ or $$F(\overline{w})=B\left(\overline{w}\right).$$  This is  equation holds if and only if  $\overline{w}$ is an equilibrium point of \eqref{hdl2}.
This completes the proof.
\end{proof}

From Proposition \ref{pro3-finite}, we derive the following direct consequence:
\begin{corollary}\label{pro4} 
	Let $h, F~:  \R^d \longrightarrow \R^d$ and  $ g~: \Omega \longrightarrow \mathbb{R}\cup \{+\infty\}$ be a proper, convex, .s.c function.  Then \(\overline{w}\in \R^d\) is a solution to GIMVIP \eqref{GIMVIP} if and only if it is an equilibrium point of the dynamical system \eqref{newydynamicalsystem-finite}. 
\end{corollary}
\begin{proof}
It follows directly from  Proposition \ref{pro3-finite} and Corollary \ref{mlh1}.  
\end{proof}
The next main result gives the finite-time stability of the dynamical system \eqref{newydynamicalsystem-finite}.
\begin{theorem}\label{tr2-fin} Let $ h~:  \R^d \longrightarrow \R^d $  be a Lipschitz continuous single-valued mapping with Lipschitz constant $\alpha$. Let   \(F: \R^d \longrightarrow \R^d\) be a single-valued mapping. Let $ g~: \Omega \longrightarrow \mathbb{R}\cup \{+\infty\}$  be a proper, convex, and lower semicontinuous function. Assume that  Conditions in Assumption  ({\bf A}) hold. 
Then,  the solution \(\overline{w}\in \R^d\) of the GIMVIP  \eqref{GIMVIP} is a globally finite-time stable equilibrium point of \eqref{newydynamicalsystem-finite} with a settling time 
$$T(w(0))\leq T_{\max}=\dfrac{1}{2^{1-p}}\dfrac{(\|w(0)-\overline{w}\|)^{2(1-p)}}{K(1-p)}$$
for some constants $K>0$ and $p \in (0.5, 1)$. 
\end{theorem}
\begin{proof} 

We define  the Lyapunov function \(V : \R^d \rightarrow \R\) as follows:
\[V(w):=\dfrac{1}{2}\|w-\overline{w}\|^2.\]
Differentiating  with respect to time of  \(V\) along the solution of \eqref{newydynamicalsystem-finite}, starting from any initial condition \(w(0)\in \R^d\setminus\{\overline{w}\}\), where $\overline{w}\in \mbox{\rm Zer}(\Xi)$ is unique with $\Xi(w)=\Xi(w), \ w\in \R^d$,   yields:
\begin{align*}
	\dot{V}=&-\left\langle w-\overline{w}, \tau  \dfrac{\Xi(w)}{\|\Xi(w)\|^{\frac{k-2}{k-1}}} \right\rangle \notag\\
	=& -\tau   \dfrac{\left\langle w-\overline{w}, \Xi(w) \right\rangle} {\|\Xi(w)\|^{\frac{k-2}{\gamma-1}}}. 
\end{align*}
 Part (iv) of Lemma \ref{bdt} yields 
\begin{align*}
	\dot{V} &\leq  -\tau \left(\rho-\Lambda\right)   \dfrac{\|w-\overline{w}\|^2} {\|\Xi(w)\|^{\frac{k-2}{k-1}}} 
\end{align*}
for all $w\in \R^d\setminus \{\overline{w}\}$. Using Part (iii) of Lemma \ref{bdt}, one has  
\begin{align}\label{eq16-finite1}
	\dot{V} \leq &  -\tau \dfrac{ \left(\rho-\Lambda\right)}{\left(\rho-\Lambda\right)^{\frac{k-2}{k-1}} }  \dfrac{\|w-\overline{w}\|^2} { \|w-\overline{w}\|^{\frac{k-2}{k-1}}}\\ \notag
    =&-\tau \left(\rho-\Lambda\right)^{\frac{1}{k-1}} \|w-\overline{w}\|^{\frac{k}{k-1}}.\label{eq16-finite1} 
\end{align}
Setting \(C:= \tau (\rho-\Lambda)^{\frac{1}{k-1}}\). Since $\rho>\Lambda, \tau>0$ it holds that   $C>0$ and inequality  \eqref{eq16-finite1} can be read as: 
\begin{align*}
	\dot{V} \leq -C \cdot  \|w-\overline{w}\|^{\frac{k}{k-1}} =- C\cdot 2^{\frac{k}{2(k-1)}} \left(\dfrac{1}{2}\|w-\overline{w}\|^2\right)^{\frac{k}{2(k-1)}}.
\end{align*}
By applying  Theorem \ref{lm1-finite} with noting that $0.5<r:=\frac{1}{2} \, \frac{k}{k-1}<1$ for all $k>2$ and $M:=C\cdot 2^{r}>0$ we obtain the conclusion. The proof is completed.
\end{proof}

 \begin{remark}\label{connection finite} Since the GIMVIP \eqref{GIMVIP} encompasses several well-known problems as special cases, the above result yields finite-time stability for these problems by choosing appropriate mappings $F$, $h$, and the function $g$. More precisely, when $h$ is the identity mapping, the result guarantees finite-time stability for IMVIPs; if, in addition, $g \equiv 0$, it applies to IVIPs. Furthermore, when $F$ is the identity mapping, finite-time stability is obtained for MVIPs, which further reduces to VIPs when $g \equiv 0$.

\end{remark}

\section{ Fixed-time and predefined-time stability  analysis} 	\label{sec3}

In this section, we investigate the fixed-time stability of solutions to the GIMVIP \eqref{GIMVIP}. For this purpose, a new dynamical system is introduced. The solutions of the generalized inverse mixed variational inequality problem are then characterized via the equilibrium points of this system. Finally, we prove the fixed-time convergence of the proposed dynamical system.


To establish the fixed-time stability of solutions to the generalized inverse mixed variational inequality problem \eqref{GIMVIP}, we construct the following dynamical system:

\begin{equation}\label{newydynamicalsystem} 
	\dot{w}=-\frac{G_d}{T_d}\tau(w)\Xi(w),
\end{equation} 
where 
\begin{equation}\label{dnrho} 
	\tau(w):=\begin{cases} 
		a_1\dfrac{1}{\|\Xi(w)\|^{1-k_1}}& +a_2\dfrac{1}{\|\Xi(w)\|^{1-k_2}}+a_3\dfrac{1}{\|\Xi(w)\|^{k_3}} \\
		& \mbox{ if } w\in  \R^d\setminus {\rm Zer}(\Xi),\\
		0 & \mbox{otherwise}, 
	\end{cases}
\end{equation} 
with  \(a_1, a_2>0, a_3> 0, k_1\in (0, 1)\), $k_3\geq 0$ and \(k_2>1\).
\begin{remark}
    When \(F\) is the identity mapping and \(k_3 = 0\), the dynamical system~\eqref{newydynamicalsystem} reduces to the system proposed in~\cite{Zheng}, which was used to establish predefined-time stability for MVIPs. Likewise, when \(F\) is the identity mapping and \(k_3 = 1\), the dynamical system~\eqref{newydynamicalsystem} coincides with the one studied in~\cite{XXJU}, where fixed-time stability for MVIPs was established.

\end{remark}
\begin{proposition}\label{pro3}
    A point $ \overline{w}\in  \R^d$ is an equilibrium point of \eqref{newydynamicalsystem}-\eqref{dnrho} if and only if it is also an equilibrium point of \eqref{hdl2}. 
\end{proposition}

\begin{proof}
Because \(\tau(w)=0\) if and only if \(w\in  \text{\rm Zer}(\Xi)\), the conclusion follows from the definition \eqref{dnrho} of $\tau$. 
\end{proof}

By combining Corollary \ref{mlh1} and Proposition \ref{pro3}, we obtain a characterization of the solutions to the GIMVIP in terms of the equilibrium points of the associated dynamical system \eqref{newydynamicalsystem}.

\begin{proposition}\label{pro4-fixed}
	A point \(\overline{w}\in \R^d\) is a solution of the GIMVIP \eqref{GIMVIP} if and only if it is an equilibrium point of the dynamical system \eqref{newydynamicalsystem}. 
\end{proposition}

The following proposition provides a sufficient condition for the existence of a solution to the system \eqref{newydynamicalsystem}.

\begin{proposition} \label{unique sol da sy}   Consider the dynamical system \eqref{newydynamicalsystem}. 
 If $\Xi: \mathbb{R}^{d} \rightarrow \mathbb{R}^{d}$ be a locally Lipschitz continuous vector field such that 
\begin{equation}  
\Xi\left(\overline{w}\right)=0 \text { and } \langle w-w^*, \Xi(w)\rangle>0, \forall w \in \mathbb{R}^{d} \backslash\left\{\overline{w}\right\}, \overline{w} \in \mathbb{R}^{d}, 
\end{equation}
 then the right-hand side of \eqref{newydynamicalsystem} is continuous for all $w \in \mathbb{R}^{n}$, and for any given initial condition, the solution of \eqref{newydynamicalsystem} exists and is uniquely determined for all $t \geq 0$.
\end{proposition}
\begin{proof}
   Using the same arguments as in the proof of Lemma 3 in \cite{Zheng}, we obtain the desired conclusion..
\end{proof}

The following lemma from \cite{Zheng} provides sufficient conditions for predefined-time stability.
\begin{lemma}\label{lmpredefined} \cite{Zheng}     
 Consider the system \eqref{JJS1}. If there exists a positive definite unbounded radial function $V: \mathbb{R}^{d} \rightarrow \mathbb{R}$, and  a user-defined parameter $T_{d}$ such that  the following conditions hold:\\
(i) $V(w(t))=0 \Leftrightarrow \dot{w}(t)=0$.\\
(ii) For any $V(w(t))>0$, there exist $c_1, c_2, c_3, r_1, r_2, T_{d}, G_{d}>0$, where $r_2>1$, and $0<r_1<1$, such that the following inequality holds:
\begin{equation}\label{pt9}
\dot{V}(w(t)) \leq-\frac{G_{d}}{T_{d}}\left(c_1 V^{r_1}(w(t))+ c_2V^{r_2}(w(t))+c_3 V(w(t))\right), 
\end{equation}
then the system \eqref{JJS1} reaches to fixed-time stability within the predefined time $T_{d}$, where 
\begin{equation}\label{pt10}
G_{d}=\frac{1}{c_3(1-r_1)} \ln \left(1+\frac{c_3}{c_1}\right)+\frac{1}{c_3(r_2-1)} \ln \left(1+\frac{c_3}{c_2}\right) 
\end{equation}

\end{lemma}
We are now in a position to state the main result of this section.
\begin{theorem}\label{tr2}   Let $ h~:  \R^d \longrightarrow \R^d $  be a Lipschitz continuous single-valued mapping with Lipschitz constant $\alpha$. Let   \(F: \R^d \longrightarrow \R^d\) be a   single-valued mapping. Let $ g~: \Omega \longrightarrow \mathbb{R}\cup \{+\infty\}$  be a proper, convex, and lower semicontinuous function. Assume that  Conditions in Assumption  ({\bf A}) hold. 

Let $\overline{w}\in \R^d$ be the unique equilibrium point of the dynamical system \eqref{newydynamicalsystem}.  Then

     For $l=1$, the solution \(\overline{w}\in \R^d\) of the GIMVIP \eqref{GIMVIP} is a fixed-time stable equilibrium point of \eqref{newydynamicalsystem} for any \(k_1\in (0, 1)\) and \(k_2>1\) and the following time estimate holds:
	$$T(w(0))\leq T_{\max}=\dfrac{1}{A_1(1-s_1)}+\dfrac{1}{A_2(s_2-1)}$$
	for some $A_1>0, A_2>0, s_1\in (0.5, 1), s_2>1$. 
	
	In addition, if we take $s_1=1-\frac{1}{2\zeta}, s_2=1+\frac{1}{2\zeta}$ with $\zeta>1$, the settling time satisfies the following:
	$$T(w(0))\leq T_{\max} =\dfrac{\pi \zeta}{\sqrt{b_1b_2}}$$
	for some constants $b_1>0, b_2>0$ and $\zeta>1$. 
    
 In particular,  if $k_3=0$, the \eqref{newydynamicalsystem} can reaches the equilibrium point within a predefined time $T_{d}$, where $T_{d}$ is a user-defined parameter, and $G_{d}= \frac{1}{a_{3} (\rho-\Lambda)\left(1-k_{1}\right)} \ln \left(1+\frac{a_{3} \sqrt{2}^{1-k_{1}} \Gamma^{1-k_{1}}}{a_{1}}\right) +\frac{1}{a_{3} (\rho-\Lambda)\left(k_{2}-1\right)} \ln \left(1+\frac{a_{3} \sqrt{2}^{1-k_{2}} (\rho-\Lambda)^{1-k_{2}}}{a_{2}}\right)$.

\end{theorem}

\begin{proof} 
 By Theorem \ref{unique sol} the generalized inverse variational inequality problem has a unique solution. By Part (iv) of Lemma \eqref{lm1}, all assumptions in Proposition \ref{unique sol da sy} are satisfied. Hence, the system \eqref{newydynamicalsystem} has a unique solution. Consider the following Lyapunov function\\
$V(w)=\frac{1}{2}\left\|w-\overline{w}\right\|^{2}$.
 Let  \(V: \R^d\rightarrow \R\) be the Lyapunov function defined by:
	\[V(w):=\dfrac{1}{2}\|w-\overline{w}\|^2.\]
  Differentiating with respect to time along the solution of \eqref{newydynamicalsystem} for Lyapunov function, starting from any \(w(0)\in  \R^d \setminus\{\overline{w}\}\) with noting that $\overline{w}\in \text{\rm Zer}(\Xi)$ being unique  one has:
	\begin{align}\label{eq12} 
		\dot{V}=&\langle w-\overline{w}, \dot{w}\rangle \notag \\
        = &- \frac{G_d}{T_d}\Big(a_1\|\Xi(w)\|^{k_1}+a_2\|\Xi(w)\|^{k_2}+a_3\Big)  \langle w-\overline{w}, \Xi(w)\rangle \notag \\
        \leq &-\frac{G_d}{T_d} \left( a_1\left(\rho-\Lambda \right) \dfrac{\|w-\overline{w}\|^2}{\|\Xi(w)\|^{1-k_1}}-a_2\big(\rho-\Lambda\big) \dfrac{\|w-\overline{w}\|^2}{\|\Xi(w)\|^{1-k_2}}-(\rho-\Lambda)a_3  \dfrac{\|w-\overline{w}\|^2}{\|\Xi(w)\|^{k_3}}  \right)
        \end{align}
	for all $w\in  \R^d\setminus \{\overline{w}\}$. 
Using Part (iii) of Lemma \ref{bdt} one obtains that 
	\begin{align}\label{nam1}
		\dot{V} \leq &\frac{G_d}{T_d}\left(-\dfrac{a_1(\rho-\Lambda)}{\Gamma^{1-k_1}}\|w-\overline{w}\|^{1+k_1}-(\rho-\Lambda)^{k_2}\Big( a_2+a_3\|\Xi(w)\|^{1-k_2-k_3}\Big) \|w-\overline{w}\|^{1+k_2} \right)\notag\\
        \leq & -C_1(k_1)\|w-\overline{w}\|^{1+k_1}-C_2(k_2, u) \|w-\overline{w}\|^{1+k_2},
	\end{align}
	where \(C_1(k_1)=\frac{G_d}{T_d}\cdot \dfrac{a_1(\rho-\Lambda)}{\Gamma^{1-k_1}}\) and \(C_2(k_2, u)= \frac{G_d}{T_d}(\rho-\Lambda)^{k_2}\Big( a_2+a_3\|\Xi(w)\|^{1-k_2-k_3}\Big)\). 
	Since $\rho>\Lambda$ and $a_i>0, i=1,2,3$ it holds that $C_1(k_1)>0, C_2(k_2, u)>0$ for all $0<r_1<1, r_2>1$. By setting $$A_i =2^{\frac{1+k_1}{2}}\frac{G_d}{T_d}\cdot \dfrac{a_1(\rho-\Lambda)}{\Gamma^{1-k_1}}$$ and choosing 
    $$A_2\in \left[2^{\frac{1+k_2}{2}} \frac{G_d}{T_d}(\rho-\Lambda)^{k_2}a_2, 2^{\frac{1+k_2}{2}} \frac{G_d}{T_d}(\rho-\Lambda)^{k_2}\Big( a_2+a_3\|\Xi(w(0))\|^{1-k_2-k_3}\Big)\right]$$
    we derive from \eqref{nam1} that 
      \begin{align}
        \dot{V} \leq &  -\Big(A_1V(w)^{s_1}+A_2V(w)^{s_2}\Big),\label{dgV} 
    \end{align}
	here  \(s_i=\frac{1+k_i}{2}\), for $i=1,2$. Note that \(A_1>0, 0.5<s_1<1\) for any \(k_1\in (0, 1)\) and \(A_2>0, s_2>1\) for any \(k_2>1\). Using Theorem \ref{lm1}, we conclude the conclusion of the theorem holds.

When $k_3=0$, from \eqref{eq12} we have 
\begin{align}\label{pt23}
\dot{V}&=\left\langle w-\overline{w}, \dot{w}\right\rangle \notag \\ 
& =-\frac{G_d}{T_d}\left\langle w-\overline{w}, a_{1} \frac{\Xi(w)}{\|\Xi(w)\|^{1-k_{1}}}+a_{2} \frac{\Xi(w)}{\|\Xi(w)\|^{1-k_{2}}}+a_{3} \Xi(w)\right\rangle  \notag \\
& \leq \frac{G_d}{T_d}\left(-a_{1} (\rho-\Lambda) \frac{\left\|w-\overline{w}\right\|^{2}}{\|\Xi(w)\|^{1-k_{1}}}-a_{2} (\rho-\Lambda) \frac{\left\|w-\overline{w}\right\|^{2}}{\|\Xi(w)\|^{1-k_{2}}}-a_{3} (\rho-\Lambda)\left\|w-\overline{w}\right\|^{2}\right) \notag \\
& \leq \frac{G_d}{T_d}\left(-\frac{a_{1} (\rho-\Lambda)}{\Gamma^{1-k_{1}}}\left\|w-\overline{w}\right\|^{1+k_{1}}-a_{2} (\rho-\Lambda) \frac{\left\|w-\overline{w}\right\|^{2}}{\|\Xi(w)\|^{1-k_{2}}}-a_{3} (\rho-\Lambda)\left\|w-\overline{w}\right\|^{2}\right) \notag \\
& \leq  \frac{G_d}{T_d}\left(-\frac{a_{1} (\rho-\Lambda)}{\Gamma^{1-k_{1}}}\left\|w-\overline{w}\right\|^{1+k_{1}}-a_{2} (\rho-\Lambda)^{k_{2}}\left\|w-\overline{w}\right\|^{1+k_{2}}-a_{3} (\rho-\Lambda)\left\|w-\overline{w}\right\|^{2}\right) \notag \\
& = \frac{G_d}{T_d}\left(-\frac{a_{1} (\rho-\Lambda)}{\Gamma^{1-k_{1}}}(2 V)^{\frac{1+k_{1}}{2}}-a_{2} (\rho-\Lambda)^{k_{2}}(2 V)^{\frac{1+k_{2}}{2}}-2 a_{3} (\rho-\Lambda) V\right) \notag\\
& =\frac{G_d}{T_d}\left(-\left(\sqrt{2}^{1+k_{1}} \frac{a_{1} (\rho-\Lambda)}{\Gamma^{1-k_{1}}} V^{\frac{1+k_{1}}{2}}+\sqrt{2}^{1+k_{2}} a_{2} (\rho-\Lambda)^{k_{2}} V^{\frac{1+k_{2}}{2}}+2 a_{3} (\rho-\Lambda) V\right)\right) \notag \\
& =\frac{G_d}{T_d}\left(-\tau_{1} V^{\frac{1+k_{1}}{2}}-\tau_{2} V^{\frac{1+k_{2}}{2}}-\tau_{3} V\right) 
\end{align}
where $\tau_{1}=\sqrt{2}^{1+k_{1}} \frac{a_{1} (\rho-\Lambda)}{\Gamma^{1-k_{1}}}>0, \tau_{2}=\sqrt{2}^{1+k_{2}} a_{2} (\rho-\Lambda)^{k_{2}}>0, \tau_{3}=2 a_{3} (\rho-\Lambda)$. Take

\begin{align*}
G_d= & \frac{2}{\tau_{3}\left(1-k_{1}\right)} \ln \left(1+\frac{\tau_{3}}{\tau_{1}}\right)+\frac{2}{\tau_{3}\left(k_{2}-1\right)} \ln \left(1+\frac{\tau_{3}}{\tau_{2}}\right) \\
= & \frac{1}{a_{3} (\rho-\Lambda)\left(1-k_{1}\right)} \ln \left(1+\frac{a_{3} \sqrt{2}^{1-k_{1}} \Gamma^{1-k_{1}}}{a_{1}}\right) \\
& +\frac{1}{a_{3} (\rho-\Lambda)\left(k_{2}-1\right)} \ln \left(1+\frac{a_{3} \sqrt{2}^{1-k_{2}} (\rho-\Lambda)^{1-k_{2}}}{a_{2}}\right).
\end{align*}
Then, by \eqref{pt23} and Lemma~\ref{lmpredefined}, we conclude that the equilibrium point of the dynamical system~\eqref{newydynamicalsystem} is reached within the predefined time \(T_d\). Moreover, for any initial condition \(w(0)\in \mathbb{R}^n\), the settling time satisfies
\[
T(w(0)) \le T_{\max} = T_d .
\]
The proof is completed. 
\end{proof}

\begin{remark}\label{rm4} 
    In line with Remark~\ref{connection finite}, when $l=1$ the above result ensures fixed-time stability for IMVIPs when $h$ is the identity mapping, and for IVIPs when $g \equiv 0$. Moreover, if $F$ is the identity mapping, fixed-time stability is obtained for MVIPs, which reduce to classical VIPs when $g \equiv 0$. When $l=0$, the above result implies the predefined stability for an equilibrium point of corresponding problems with a particular choice of mapping $F, h$ and the function $g$. 

When $k_3=1$ and $F$ is the identity mapping, the resulting dynamical system  \eqref{newydynamicalsystem} coincides with the one studied in \cite{XXJU}. Hence, as shown in \cite{XXJU}, if $h \equiv \nabla f$ for some differentiable function $f$, the system \eqref{newydynamicalsystem} yields fixed-time stability for the constrained optimization problem
\begin{equation}\label{Cop}
    \min_{w\in \R^d} f(w)+g(w),
\end{equation}
under suitable assumptions on $f$ and $g$. Note that $g$ plays the role of a constraint; in particular, if $g \equiv 0$, \eqref{Cop} reduces to an unconstrained optimization problem, which can be viewed as a special case of the GIMVIP \eqref{GIMVIP}.

Furthermore, following \cite{XXJU}, let $\mathcal{L}:\R^d \times \R^m \to \R$ be differentiable and define
\[
h(w,v) = \big(\nabla_w \mathcal{L}(w,v), -\nabla_v \mathcal{L}(w,v)\big).
\]
Under suitable assumptions on $\mathcal{L}$, fixed-time stability is obtained for the minimax problem
\[
\inf_{w\in \R^d}\sup_{v\in \R^m} \mathcal{L}(w,v).
\]

\end{remark}
    

\section{ Discretization of the dynamical system} \label{sec4}
This section is devoted to the iterative method. By employing a forward difference scheme for time discretization, we derive an iterative algorithm for the GIMVIP and show that fixed-time stability is preserved. We first recall a result in a more general setting. 
\subsection{Discretization of the finite-time dynamical system}
We first apply a forward discretization to the system~\eqref{newydynamicalsystem-finite} with step size $\Delta t_n$, which yields the following iterative algorithm.
\begin{equation*} 
    \dfrac{w(t_{n+1})-w(t_n)}{\Delta t_n}=-\tau  \dfrac{\Xi\left(w(t_n)\right)}{\|\Xi\left(w(t_n)\right)\|^{\frac{k-2}{k-1}}},  k>2, \tau>0.
\end{equation*}
or 
\begin{equation}\label{alg1}
w(t_{n+1})=w(t_n) -\tau\Delta t_n  \dfrac{\Xi\left(w(t_n)\right)}{\|\Xi\left(w(t_n)\right)\|^{\frac{k-2}{k-1}}},  k>2, \tau>0. 
\end{equation}
Specially when $\Delta t_n=\theta$ fixed step size Algorithm \eqref{alg1} becomes
\begin{equation}\label{alg2}
  w(t_{n+1})=w(t_n) -\tau\theta   \dfrac{\Xi\left(w(t_n)\right)}{\|\Xi\left(w(t_n)\right)\|^{\frac{k-2}{k-1}}}, k>2, \tau>0. 
\end{equation}
Setting $w(t_n)=w_n$ for all $n=1,2\dots,$ \eqref{alg2} reads as
\begin{equation}\label{alg2}
  w_{n+1}=w_n -\tau\theta   \dfrac{\Xi\left(w_n\right)}{\|\Xi\left(w_n\right)\|^{\frac{k-2}{k-1}}}, k>2, \tau>0. 
\end{equation}
\subsection{Discretization of the finite-time dynamical system}

To establish the preservation of fixed-time stability under discretization,
we study a more general framework by considering the differential inclusion
\begin{equation}\label{eq20}
	\dot{w} \in \Upsilon(w),
\end{equation}
and impose the following assumption.

    \noindent ({\bf B}) 	Let  \(\Upsilon: \R^d \longrightarrow 2^{\R^d}\) be an upper semi-continuous set-valued mapping such that its values are non-empty, convex, compact, and  \(0\in \Upsilon(\overline{w})\) for some \(\overline{w}\in  \R^d\). Assume  that there exists a positive definite, radially unbounded, locally Lipschitz continuous and regular function \(V: \R^d\rightarrow \R^d\) such that \(V(\overline{w})=0\) and 
	\begin{equation*}
		\sup \dot{V}(w)\leq -\left( A_1 V(w)^{1-\frac{1}{\chi}}+A_2 V(w)^{1+\frac{1}{\chi}}\right) 
	\end{equation*}
	for all \(w\in   \R^d\setminus\{\overline{w}\}\), with \(A_1, A_2>0\) and \(\chi>1\), here  
	\begin{equation*}
		\dot{V}(w)=\left\{ z\in \R: \exists w\in \Upsilon (w) \mbox{ such that } \langle v, w\rangle =z, \forall v\in \partial_cV(w)\right\}
	\end{equation*}
	Here  \(\partial_c V(w)\) is Clarke's generalized gradient of the function \(V\) at the point \(w\in   \R^d\).
    
\noindent Using forward  discretization  of \eqref{eq20} with step size $\Delta t_n>0$  we obtain 
\begin{equation*}
    \dfrac{w(t_{n+1})-w(t_n)}{\Delta t_n} = \Upsilon\left(w(t_n) \right)
\end{equation*}
Choosing $\Delta t_n=\theta$ and setting $w(t_n)=w_n$ for all $n=1,2,\dots$ we get 
\begin{equation}\label{eq24} 
	w_{n+1}\in w_n+\theta \Upsilon(w_n). 
\end{equation}

We recall the following result from \cite{Garg21}, which provides an upper bound for the error of the sequence generated by \eqref{eq24}.
\begin{theorem}\cite{Garg21}\label{tr4}  Let $\overline{w}$  is the equilibrium point of \eqref{eq20}. 
	Assume that the conditions in Assumption ({\bf B}) holds and  the function \(V\) also satisfies the following quadratic growth condition
	\begin{equation*}
		V(w)\geq c\|w-\overline{w}\|^2 
	\end{equation*}
	for every $w\in  \R^d$, and for some $c>0$. 
	Then, for all \(w_0\in  \R^d\) and \(\eps>0\), there exists 
	 $\theta^*>0$ such that for any \(\theta\in (0, \theta^*]\), one has:
	\begin{equation*}
		\|w_n-\overline{w}\| <\begin{cases}
			\frac{1}{\sqrt{c}} \left( \sqrt{\frac{A_1}{A_2}}\tan\left(\frac{\pi}{2}-\frac{\sqrt{A_1A_2}}{\chi}\theta n\right)\right)^{\frac{\chi}{2}}+\eps &\mbox{ if }\ n\leq n^*,\\
			\eps & \mbox{otherwise},
		\end{cases}
	\end{equation*}
	where \(w_n\) is a solution  of \eqref{eq24} starting from the point \(w_0\)  and $n^* = \left \lceil \dfrac{\chi\pi}{2\lambda\sqrt{A_1A_2}} \right \rceil$.
\end{theorem}

Using this result, we obtain the fixed-time stability for the discretization version of \eqref{newydynamicalsystem} and get an upper bound for settling time. 
\begin{theorem}\label{dl cuoi}
	Consider the forward discretization version of \eqref{newydynamicalsystem}:
	\begin{equation}\label{eq29} 
		w_{n+1}=w_n-\theta \frac{G_d}{T_d}\tau(w_n)\Xi(w_n),
	\end{equation}
	where  \(\tau\) is given by \eqref{dnrho},  \(a_1, a_2>0, a_3>0, k_1=1-2/\chi\) and \(k_2=1+2/\chi\), \(\chi\in (2, \infty)\), and \(\theta>0\) is the time-step. Then, for every \(w_0\in  \R^d\), every \(\eps>0\), if  Assumption {\bf (A)} holds, there exist \(\chi>2, A_1, A_2>0\) and \(\theta^*>0\) such that for any \(\theta\in (0, \theta^*]\), we have: 
	\begin{equation*}
		\|w_n-\overline{w}\|< \begin{cases}
			\sqrt{2}\left( \sqrt{\frac{A_1}{A_2}}\tan\left(\frac{\pi}{2}-\frac{\sqrt{A_1A_2}}{\chi}\lambda n\right)\right)^{\frac{\chi}{2}}+\eps &  \mbox{ if } n\leq n^*,\\ 
			\eps & \mbox{otherwise},
		\end{cases}
	\end{equation*}
	where \(n^*=\left\lceil \dfrac{\chi\pi}{2\lambda \sqrt{A_1A_2}} \right\rceil\) and \(w_n\) is determined by \eqref{eq29} starting from the point \(w_0\) and \(\overline{w}\in  \R^d\) is the unique solution of the GIMVIP \eqref{GIMVIP}. 
\end{theorem}
\begin{proof} We observe from the proof of Theorem \ref{tr2}  that  inequality \eqref{dgV} is valid for any \(s_1\in  (0, 1)\)  and \(s_2>1\). Also, since \( s_1 =1-\dfrac{2}{\chi}\) and \(s_2 =1+\dfrac{2}{\chi}\), it holds for  any \(\chi>2\). Hence, all conditions in Theorem \ref{tr4} are fulfilled with the Lyapunov function $V(w)=\dfrac{1}{2}\|w-\overline{w}\|^2$ and with $c=\frac{1}{2}$. Applying Theorem~\ref{tr4} yields the desired conclusion.
\end{proof}


\section{Numerical  illustration}\label{num}
In this section, we present a numerical example to illustrate the convergence behavior of the iterative algorithm \eqref{eq29}. To this end, we reuse Example~\ref{ex1}. The numerical simulations were implemented in Python using Google Colab.

\begin{example}
Let $d=1$, $\Omega=[0,+\infty)$, and define
\[
h(w)=\tfrac{w}{2}, \qquad F(w)=\tfrac{3w}{4}, \qquad g(w)=w^2+2w+1, \quad w\in\R.
\]
As shown in Example~\ref{ex1}, the mappings $h$ and $F$, together with the function $g$, satisfy Assumption~\textbf{(A)} with
\[
\alpha=\lambda=\tfrac{1}{2}, \qquad \beta=\tfrac{3}{4}, \qquad \rho=\tfrac{4}{3}, \qquad \mu=\tfrac{3}{8}.
\]

Algorithm~\eqref{eq29} is implemented with the parameters
\(\gamma = 1\), \(a_1 = 0.9\), \(a_2 = 0.5\), \(a_3 = 10^{-4}, G_d=T_d=1\),
\(k_1 = 0.4\), \(k_2 = 1.5\), and
\(\theta = 10^{-4} + \frac{1}{n}\), where \(n\) denotes the iteration index.
Algorithm~\eqref{alg2} is implemented with the parameters
\(\tau = 1\) and \(\theta = 0.2\).

Starting from the initial value \(w_0 = 50.0\), after \(150\) iterations the sequence generated by~\eqref{eq29} converges to the approximate solution
\(-5.33 \times 10^{-5}\) when \(k_3 = 1\), and to \(-1.13 \times 10^{-4}\) when \(k_3 = 0\).
In contrast, the sequence generated by Algorithm~\eqref{alg2} converges to the approximate value \(2.08\) for \(k = 3\), and to \(1.39 \times 10^{-3}\) for \(k = 2\).

\end{example}
\begin{figure}[H]
    \centering
    \includegraphics[width=1.\linewidth]{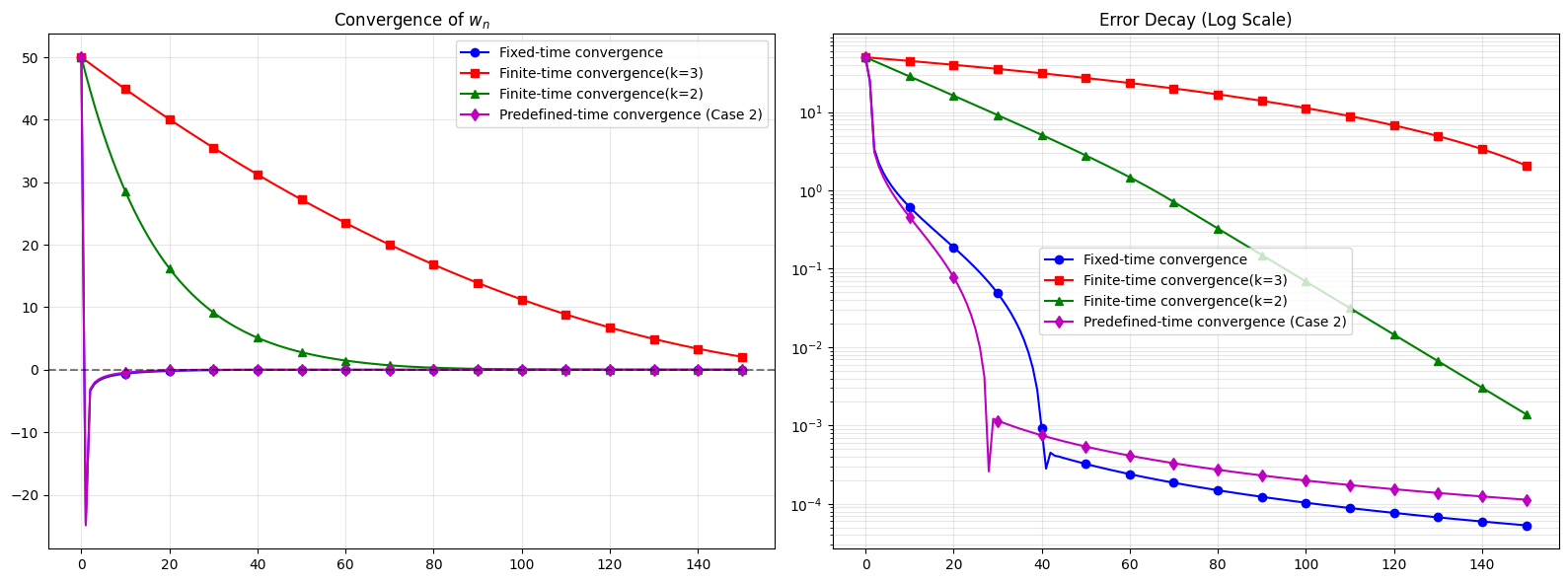}
    \caption{Convergence rate of Algorithm \ref{eq29}.}
    \label{fig:nam}
\end{figure}

\begin{remark}
  By Remark~\ref{rm4}, the dynamical system proposed in~\cite{XXJU} is a special case of~\eqref{newydynamicalsystem} when \(F\) is the identity mapping and \(k_3 = 1\). Consequently, the numerical examples presented in~\cite{XXJU} illustrate the effectiveness of our proposed algorithm in this setting. Likewise, the differential equation system studied in~\cite{Zheng} coincides with~\eqref{newydynamicalsystem} when \(F\) is the identity mapping and \(k_3 = 0\); hence, the numerical examples therein can also be used to demonstrate the performance of our method. Moreover, when \(a_3 = 0\) and \(F\) is the identity mapping, system~\eqref{newydynamicalsystem} reduces to the one investigated in~\cite{GAR}, and the corresponding numerical examples further validate the effectiveness of our approach.

\end{remark}

\section{Conclusion}\label{sec6}

This work has developed three dynamical-system-based approaches for solving generalized inverse mixed variational inequality problems. 
Specifically, we proposed continuous-time dynamical systems that achieve finite-time, fixed-time, and predefined-time stability, respectively. 
The finite-time stable system ensures convergence within a finite settling time that depends on the initial condition, whereas the fixed-time stable system guarantees convergence within a uniform time bound independent of the initial state. 
More importantly, the predefined-time stable system allows the convergence time to be explicitly prescribed in advance through appropriate parameter selection, providing greater flexibility and user autonomy in time-critical applications.

For the fixed-time and predefined-time stable models, we established the existence and uniqueness of solutions to the corresponding GIMVIPs and proved that the generated trajectories converge to the unique solution within the specified time bounds. 
Furthermore, it was shown that the convergence properties of the continuous-time dynamics are preserved under an explicit forward Euler discretization, leading to a proximal point-type iterative algorithm with guaranteed fixed-time and predefined-time convergence. 
Numerical experiments were presented to validate the theoretical findings and to demonstrate the effectiveness and accelerated convergence behavior of the proposed methods.

Several directions for future research merit further investigation. 
These include extending the proposed finite-time, fixed-time, and predefined-time stability analyses to more general settings, such as infinite-dimensional Hilbert or Banach spaces, as well as relaxing the assumptions imposed on the underlying operators and model parameters. 
Another promising avenue is the development of adaptive or data-driven parameter selection strategies for predefined-time stable algorithms, which could further enhance their applicability in large-scale optimization and inverse equilibrium problems.

\section*{Declarations} 
{\bf Conflict of interest} The authors declare no competing interests.

\end{document}